\documentclass[12pt]{article}
\usepackage[T2A]{fontenc}
\usepackage[utf8]{inputenc}
\usepackage{amssymb,amsmath,amsfonts,amsthm,amscd,latexsym,verbatim,indentfirst,hyperref}
\hypersetup{hidelinks}

\usepackage{ulem}

\usepackage{stmaryrd}
\usepackage{xcolor}
\usepackage{geometry}
\def\sl{\mathrm{sl}}

\def\Cur{\mathrm{Cur}}

\def\Aut{\mathrm{Aut}}

\newtheorem{remark}{Remark}
\newtheorem{lemma}{Lemma}
\newtheorem{theorem}{Theorem}

\newtheorem{corollary}{Corollary}

\begin{document}


\begin{center}
{\Large
Conformal Yang---Baxter equation on $\Cur(\sl_2(\mathbb{C}))$
}

\smallskip

Vsevolod Gubarev, Roman Kozlov
\end{center}

\begin{abstract}
In 2008, J. Liberati defined what is a conformal Lie bialgebra and introduced the conformal classical Yang---Baxter equation (CCYBE). An $L$-invariant solution to the weak version of CCYBE provides a conformal Lie bialgebra structure.
We describe all solutions to the conformal classical Yang---Baxter equation
on the current Lie conformal algebra $\Cur(\sl_2(\mathbb{C}))$ and to the weak version of it.

\end{abstract}

\section{Introduction}

The theory of Lie bialgebras and Poisson---Lie groups was found by V.G. Drinfeld~\cite{Belavin1982}, M.A. Semenov-Tian-Shansky~\cite{Semenov1983} and others at the early 1980s.
This notion led further V.G.~Drinfeld to define quantum groups, when the result of M.A. Semenov-Tian-Shansky is applicable for integrable systems and the Lax equation.
We refer to the known books on the subject~\cite{Guide,Majid}.

Lie conformal algebras introduced by V.G. Kac in~\cite{Kac} encode the singular parts of the operator product expansion (OPE) of chiral fields in conformal field theory developed by A.A.~Belavin, A.M. Polyakov and A.B. Zamolodchikov~\cite{BPZ} in 1984.
The structure theory of Lie and associative conformal (super)algebras was studied in~\cite{SimpleConf,Fattori,Kolesnikov}.

Pseudoalgebras defined by B. Bakalov, A. D’Andrea and V.G.~Kac in~\cite{BAK} generalize conformal algebras involving a cocommutative Hopf algebra~$H$. For $H = \{e\}$, we get ordinary algebras and for $H = F[\partial]$ we obtain exactly conformal algebras.

In 2008, J. Liberati introduced~\cite{Liberati} a notion of Lie conformal bialgebra and suggested so called conformal classical Yang---Baxter equation (CCYBE) as a source of (coboundary) Lie conformal bialgebras.
In 2012, these notions were generalized for Lie pseudobialgebras~\cite{BL}. 

We also mention the works devoted to the analogous notions defined for associative conformal algebras and associative pseudoalgebras.
In~\cite{Bai2021} and in~\cite{Liu2022} the associative Yang---Baxter equation was defined for associative conformal algebras and associative pseudoalgebras respectively. 
In~\cite{Bai}, so called conformal $S$-equation for left-symmetric conformal algebras was written down and studied.

In~\cite{SimpleConf}, A. D'Andrea and V.G. Kac proved the following structure result: every simple conformal algebra of finite type
is isomorphic either to the Virasoro conformal algebra $\mathrm{Vir}$ or to the current Lie conformal algebra $\Cur(g)$ associated to a simple finite-dimensional Lie algebra $g$.

In the current work, we classify all solutions to (weak) conformal classical Yang---Baxter equation on $\mathrm{Vir}$ and $\Cur(\sl_2(\mathbb{C}))$.
For $L = \Cur(\sl_2(\mathbb{C}))$, we obtain sufficient and necessary conditions of a tensor $r = \sum \nolimits A_{ql}(\partial_{\otimes 1},\partial_{\otimes 2}) q\otimes l\in L\otimes L$, where $q, l \in \{e, f, h\}$, to be an $L$-invariant solution to the (weak) CCYBE in terms of the polynomials $A_{ql}(x,-x)$. In particular, we show that $A_{ql}(x,-x) - A_{ql}(0,0)$ are odd polynomials on~$x$, which are also proportional to each other. 
The same approach applied for the simple Lie conformal algebra $L = \mathrm{Vir}$ implies that $r = A(\partial_{\otimes 1},\partial_{\otimes 2}) v\otimes v$ is a~solution to the weak CCYBE if and only if $A(x, -x) = 0$.

In~\cite[Ex.\,2.18]{Liberati}, coboundary conformal Lie bialgebra structures of special type on the solvable Lie conformal algebra of rank~2 are considered.
In the case, the condition that the involving polynomial has to be odd also appears.

Now, we give a~short outline of the work.
In~\S2, we provide preliminaries about conformal Lie algebras and the conformal classical Yang---Baxter equation.

In~\S3, we get the complete description of solutions to the (weak) CCYBE on $L = \Cur(\sl_2(\mathbb{C}))$.
For this, we clarify the situation with solutions to the classical Yang---Baxter equation on $\sl_2(\mathbb{C})$ (\S3.1). Further, for a tensor $r = \sum \nolimits A_{ql}(\partial_{\otimes 1},\partial_{\otimes 2}) q\otimes l\in L\otimes L$, $q, l \in \{e, f, h\}$,
we rewrite what $L$-invariance means in terms of the polynomials $A_{ql}(x,-x)$ (\S3.2).
The paragraph~\S3.3 is the most technical part of the work,
where we interpret CCYBE and the weak version of it again in terms of $A_{ql}(x,-x)$.  
Finally, we get the complete description of solutions to the weak conformal classical Yang---Baxter equation 
and to CCYBE itself on $\Cur(\sl_2(\mathbb{C}))$ in Theorem~\ref{thm} and in Corollary~\ref{coro:main}. 

In~\S4, we get the description of solutions to the weak CCYBE
on $\mathrm{Vir}$ (Theorem~\ref{Vir}).

\section{Preliminaries}
Throughout this paper let us fix the standard basis $e,f,h$ of $\sl_2(\mathbb{C})$ such that
$$
[e,f] =  h, \quad
[h,e] = 2e, \quad
[h,f] = -2f.
$$

\subsection{Conformal Lie algebras}

A (free) $\mathbb {C}[\partial]$-module $L$ is called a conformal algebra
if there is a $\lambda$-bracket on $L$
$$
[\cdot_\lambda\cdot]\colon L\otimes L\to \mathbb{C}[\lambda]\otimes L
$$
satisfying the identities,
\begin{equation} \label{Conf:sesqui}
[\partial a_\lambda b] = -\lambda [a_\lambda b],\quad
[a_\lambda\partial b]=(\lambda +\partial)[a_\lambda b].
\end{equation}

A conformal algebra~$L$ is called a conformal Lie algebra if $L$ satisfies
the following conformal analogues of anticommutativity and the Jacobi identity,
$$
[a_\lambda b]=-[b_{-\lambda-\partial}a], \quad
[a_\lambda[b_\mu c]]
 - [b_\mu[a_\lambda c]]
 = [[a_\lambda b]_{\lambda+\mu}c].
$$

The Virasoro Lie conformal algebra $\mathrm{Vir}$ is defined as follows,
$$
\mathrm{Vir} = \mathbb{C}[\partial]v,\quad 
[v_{\lambda}v] = (\partial + 2\lambda)v.
$$

Given a Lie algebra~$g$, the current Lie conformal algebra $\Cur(g)$
on the space $\mathbb{C}[\partial]g$ is defined as follows,
$$
[f(\partial)a_\lambda g(\partial)b]
 = f(-\lambda)g(\lambda+\partial)[a,b], \quad a,b\in g.
$$

A (left) module $M$ over a Lie conformal algebra $L$ is a left
$\mathbb{C}[\partial]$-module endowed with a $\mathbb{C}$-linear map 
$(\cdot_\lambda\cdot)\colon C\otimes M \to M[\lambda]$ satisfying the identities
$$
\partial a_\lambda v = - \lambda a_\lambda v, \quad 
a_\lambda \partial v = (\partial+\lambda) a_\lambda v, \quad
[a_\lambda b]_{\lambda + \mu} v = a_\lambda (b_\mu v) - b_\mu (a_\lambda v)
$$
for all $a, b \in L$, $v \in M$.

\begin{remark}
Given a conformal Lie algebra~$L$, the space $L^{\otimes n}$ 
is a left $L$-module under the action
$$
a_\lambda (a_1 \otimes \dots \otimes a_n) 
 = \sum\limits_{i=1}^n a_1 \otimes \dots \otimes [a_\lambda a_i] \otimes \dots \otimes a_n,
$$
where $a, a_1, \dots a_n \in L$.
\end{remark}

A~$\mathbb{C}[\partial]$-submodule $I$ of $L$ is called an ideal of $L$ if
$[I_\lambda L]\subset I$.
A conformal Lie algebra~$L$ is of finite type if $L$ is finitely-generated as $\mathbb{C}[\partial]$-module.
A conformal Lie algebra~$L$ is called simple if $[L_\lambda L]\neq(0)$
and there are only two ideals of $L$: $(0)$ and $L$.

Recall~\cite{SimpleConf} that every simple conformal algebra of finite type
is isomorphic either to $\mathrm{Vir}$ or to $\Cur(g)$
associated to a simple finite-dimensional Lie algebra $g$.

Given a Lie algebra~$g$ and an automorphism~$\varphi$ of $g$, we may extend
it to an automorphism of $\Cur(g)$ as a $\partial$-linear operator by the formula
$\varphi(f(\partial)a) = f(\partial)\varphi(a)$, $a\in g$.

\subsection{Conformal classical Yang---Baxter equation}

Let $L$ be a Lie conformal algebra and 
$r = \sum\nolimits a_i \otimes b_i \in L\otimes L$. Set 
$\partial_{\otimes 1} = \partial \otimes 1 \otimes 1$, 
$\partial_{\otimes 2} = 1 \otimes \partial \otimes 1$, 
$\partial_{\otimes 3} = 1 \otimes 1\otimes \partial$, and 
$\partial^{\otimes 3} =  \partial_{\otimes 1} + \partial_{\otimes 2} + \partial_{\otimes 3}$. 
The following equation,
\begin{multline} \label{CCYBE}
\llbracket r, r\rrbracket 
 := \sum\nolimits ([{a_i}_\lambda a_j] \otimes b_i \otimes b_j
\rvert_{\lambda = \partial_{\otimes 2}} 
- a_i \otimes [{a_j}_\lambda b_i] \otimes b_j \rvert_{\lambda 
= \partial_{\otimes 3}} \\
- a_i \otimes a_j \otimes [{b_j}_\lambda b_i] 
\rvert_{\lambda = \partial_{\otimes 2}}) = 
0~(\!\!\!\!\!\!\mod~\partial^{\otimes 3}),
\end{multline}
holding in $L^{\otimes 3}$, is called the~\textit{conformal classical Yang---Baxter equation} (CCYBE)~\cite{Liberati}.

A~tensor~$r\in L\otimes L$ satisfying the~\eqref{CCYBE} is called a solution to the CCYBE.

An equation 
\begin{equation} \label{weak-CCYBE}
a_\mu \llbracket r, r\rrbracket = 0\ (\!\!\!\!\!\!\mod \mu = - \partial^{\otimes 3})
\end{equation}
fulfilled for all $a \in L$ is called the {\it weak CCYBE}.

A solution $r$ to the CCYBE (or the weak one) is called {\it $L$-invariant}, 
if the following equality holds for every $a\in L$,
\begin{equation} \label{ConfInv}
a_{\lambda}(r + \tau(r))\rvert_{\lambda = - \partial^{\otimes 2}} = 0,
\end{equation}
where $\partial^{\otimes2} = \partial\otimes 1 + 1 \otimes \partial$
and $\tau\colon L\otimes L\to L\otimes L$ is defined as follows,
$\tau(a\otimes b) = b\otimes a$. 
A solution~$r$ of the (weak) CCYBE is called {\it skew-symmetric} if $r+\tau(r) = 0$.

In~\cite{Liberati}, J. Liberati proved that given a conformal Lie algebra~$L$ and $r\in L\otimes L$, the map
$\delta(a) = a_\lambda r|_{\lambda = -\partial^{\otimes2}}$
is the cocommutator of a conformal Lie bialgebra structure on~$L$ if and only if $r$ is an $L$-invariant solution to the weak CCYBE on~$L$.
Hence, we focus our attention on skew-symmetric and $L$-invariant solutions
to the (weak) conformal classical Yang---Baxter equation.

\begin{remark}
Let $L = \Cur(g)$ be a conformal Lie algebra, $r= \sum a_i\otimes b_i\in L\otimes L$ be a~solution to the (weak) CCYBE, and $\varphi\in\Aut(g)$ be extended to the automorphism of~$L$. 
Then $\varphi(r) := \sum \varphi(a_i)\otimes \varphi(b_i)$ is again a~solution to the (weak) CCYBE. Indeed, for a~solution~$r$ to the CCYBE we have
\begin{multline*}
\llbracket\varphi(r), \varphi(r)\rrbracket 
 = \sum \nolimits ([{\varphi(a_i)}_\lambda \varphi(a_j)] \otimes 
 \varphi(b_i) \otimes \varphi(b_j) \rvert_{\lambda = \partial_{\otimes 2}} \\
 - \varphi(a_i) \otimes [{\varphi(a_j)}_\lambda \varphi(b_i)] \otimes \varphi(b_j) \rvert_{\lambda 
 =  \partial_{\otimes 3}} 
 - \varphi(a_i) \otimes \varphi(a_j) \otimes [{\varphi(b_j)}_\lambda \varphi(b_i)]\rvert_{\lambda=\partial_{\otimes 2}}) \\
 = \sum \nolimits ([{\varphi(a_i)}, \varphi(a_j)] 
  \otimes \varphi(b_i) \otimes \varphi(b_j) \rvert_{\lambda = \partial_{\otimes 2}} \\
 - \varphi(a_i) \otimes [{\varphi(a_j)}, \varphi(b_i)] 
 \otimes \varphi(b_j) \rvert_{\lambda=\partial_{\otimes 3}} 
 - \varphi(a_i) \otimes \varphi(a_j) \otimes [{\varphi(b_j)}, \varphi(b_i)] 
 \rvert_{\lambda = \partial_{\otimes 2}}) \allowdisplaybreaks \\
 = \sum \nolimits (\varphi([{a_i}, a_j] \otimes 
 \varphi(b_i) \otimes \varphi(b_j) \rvert_{\lambda = \partial_{\otimes 2}} \\
 - \varphi(a_i) \otimes \varphi([{a_j}, b_i] 
 \otimes \varphi(b_j) \rvert_{\lambda=\partial_{\otimes 3}} 
 - \varphi(a_i) \otimes \varphi(a_j) \otimes \varphi([{b_j}, b_i]\rvert_{\lambda = \partial_{\otimes 2}}) \\
 = \varphi\bigg(
 \sum\nolimits ([{a_i}_\lambda a_j] \otimes b_i \otimes b_j\rvert_{\lambda = \partial_{\otimes 2}} 
 - a_i \otimes [{a_j}_\lambda b_i] \otimes b_j \rvert_{\lambda=\partial_{\otimes 3}} \\
 - a_i \otimes a_j \otimes [{b_j}_\lambda b_i]\rvert_{\lambda=\partial_{\otimes 2}})\bigg) 
 = \varphi( \llbracket r, r \rrbracket) 
 = 0\ (\!\!\!\!\!\!\mod \partial^{\otimes 3}).
\end{multline*}
If $r$ is a solution to the weak CCYBE, it is enough to consider $\varphi(\varphi^{-1}(a))_\lambda \llbracket \varphi(r),\varphi(r) \rrbracket$ for all $a\in \Cur\,g$.
\end{remark}

\section{Solutions to the (weak) CCYBE on $\Cur(\sl_2(\mathbb{C}))$}

Recall that $r = \sum \nolimits a_i\otimes b_i$, where $a_i, b_i \in
\Cur (\sl_2(\mathbb{C}))$ are polynomials on $\partial$ with coefficients in
$\sl_2(\mathbb{C})$. 
It is clear that we may write
\begin{multline*}
r = \sum \nolimits (g^e_i(\partial) e + g^f_i(\partial) f + g^h_i(\partial) h)
\otimes (\hat{g}^e_i(\partial) e + \hat{g}^f_i(\partial) f + \hat{g}^h_i(\partial) h)
\\
= \sum \nolimits (g^e_i(\partial) e \otimes \hat{g}^e_i(\partial) e) + \dots.
\end{multline*}

Since a group of terms may not necessarily be presented
as a decomposable tensor, it is convenient to denote all the stuff with $\partial$ 
as polynomial on
two independent variables $\partial_{\otimes 1}$ and $\partial_{\otimes 2}$.

Thus, we may write a
$\Cur(\sl_2(\mathbb{C}))$-invariant solution to either version of the CCYBE in the form 
$r = \sum \nolimits A_{ql}(\partial_{\otimes 1},\partial_{\otimes 2}) q\otimes l$, where $q, l \in \{e, f, h\}$ and $A_{ql}(x, y)\in\mathbb{C}[x,y]$.

First of all, in this section we recall the general solution to the CYBE on the Lie algebra $\sl_2(\mathbb{C})$.
After that, we get the conditions fulfilled on the polynomial coefficients of~$r$.
Finally, we obtain the exact form of the solutions~$r$.

\subsection{Case of $\sl_2(\mathbb{C})$}

The CCYBE on $\Cur(\sl_2(\mathbb{C}))$ for $\partial = 0$ converts to the ordinary (non-conformal) CYBE.
We need the following result.

\begin{lemma} \label{lem:ZeroLevel}
a) An $\sl_2(\mathbb{C})$-invariant solution to the weak CYBE has the form
\begin{equation} \label{gen-sol}
r_0 = \alpha (h \otimes e {-} e \otimes h) 
 + \beta (f \otimes e {-} e \otimes f) + 
\gamma (h \otimes f {-} f \otimes h) 
 + \zeta (h \otimes h {+} 4 e \otimes f),
\ \alpha, \beta, \gamma, \zeta \in \mathbb{C}.
\end{equation}
The general form of the skew-symmetric solution to the weak CYBE comes when $\zeta = 0$.

b) \cite{Stolin} Up to the action of~$\Aut(\sl_2(\mathbb{C}))$ and multiply on a nonzero scalar, a skew-symmetric solution to CYBE on~$\sl_2(\mathbb{C})$ equals~$h\otimes e - e\otimes h$.
\end{lemma}

\begin{proof}
a) Let $r$ be an $L$-invariant solution to the weak CYBE, where $L = \sl_2(\mathbb{C})$. Denote 
$$
r_0 = e\otimes (a_e e + a_f f + a_h h)
   + f\otimes (b_e e + b_f f + b_h h)
   + h\otimes (c_e e + c_f f + c_h h)
$$
for some $a_x,b_x,c_x\in\mathbb{C}$. Then 
\begin{multline*}
r_0 + \tau(r_0)
 = 2a_e e\otimes e 
 + (a_f+b_e)(e\otimes f + f\otimes e)
 + (a_h+c_e)(e\otimes h + h\otimes e) \\
 + 2b_f f\otimes f
 + (b_h + c_f)(f\otimes h + h\otimes f)
 + 2c_h h\otimes h.
\end{multline*}
The action of $e,f,h$ on $r_0+\tau(r_0)$ give us equalities
$$
a_e = 0, \quad
b_f = 0, \quad
a_h+c_e = 0, \quad
b_h+c_f = 0, \quad
a_f+b_e-4c_h = 0.
$$
Thus, 
\begin{multline*}
r_0 = a_h (e \otimes h - h \otimes e) 
  + b_e (f \otimes e - e \otimes f) 
  + b_h (f \otimes h - h \otimes f) 
  + c_h (h \otimes h + 4 e \otimes f) \\
  = a_h (e \otimes h - h \otimes e) 
  + (b_e-2c_h)(f \otimes e - e \otimes f) 
  + b_h (f \otimes h - h \otimes f) 
  + c_h (h \otimes h + 2 e \otimes f + 2f\otimes e).
\end{multline*}

Denote $a_h$, $b_e-c_h$, $b_h$, and $c_h$ as $\alpha$, $\beta$, $\gamma$, and $\zeta$, respectively.
We may also present 
$r_0 = s + r_0^{-}$,
where 
$s = \zeta (h\otimes h + 2(e\otimes f + f\otimes e))$ and
$r_0^{-}$ is a skew-symmetric part of $r_0$.
Since $s$ is $L$-invariant,
we have, due to~\cite[p.\,54]{Guide}, that 
$$
CYBE( s + r_0^- ) = CYBE(s) + CYBE(r_0^-).
$$
It is known that $CYBE(s) = 0$ and $CYBE(r_0^-) \in\wedge^3 L$,
i.\,e., $CYBE(r_0^-)$ is constant and, hence, $L$-invariant.

The second part of a) follows trivially.
\end{proof}

\subsection{$L$-invariance} \label{sec:Skew-sym}

Let us substitute $a = g(\partial) e$, where $g(\partial)\in\mathbb{C}[\partial]\setminus\{0\}$, in~\eqref{ConfInv}. 
Modulo $\lambda = - \partial^{\otimes 2}$, we have
\begin{multline*} 
0 = a_{\lambda} (r + \tau (r)) 
 = g(\partial) e_{\lambda} \big(\sum\nolimits 
(A_{ql}(\partial_{\otimes 1}, \partial_{\otimes 2}) q\otimes l + 
A_{ql}(\partial_{\otimes 2}, \partial_{\otimes 1}) l\otimes q)\big) \\
= \sum\nolimits (g(-\lambda)A_{ql}((\partial+\lambda)_{\otimes 1}, 
\partial_{\otimes 2})[e_{\lambda} q] \otimes l + g(-\lambda) 
A_{ql}(\partial_{\otimes 1}, 
(\partial+\lambda)_{\otimes 2}) q\otimes [e_{\lambda} l] \\
+ g(-\lambda) A_{ql}(\partial_{\otimes 2}, 
(\partial+\lambda)_{\otimes 1}) [e_{\lambda} l]\otimes q +
g(-\lambda) A_{ql}((\partial+\lambda)_{\otimes 2}, 
\partial_{\otimes 1}) l\otimes [e_{\lambda} q]).
\end{multline*}

Applying the multiplication table of $\Cur(\sl_2(\mathbb{C}))$, we get modulo $\lambda = - \partial^{\otimes 2}$,
\begin{multline*} 
0 = g(-\lambda)\sum\limits_{q,l} ((A_{fl}((\partial+\lambda)_{\otimes 1}, 
\partial_{\otimes 2}) h \otimes l - 2 A_{hl}((\partial+\lambda)_{\otimes 1}, 
\partial_{\otimes 2}) e \otimes l) \\
+ (A_{qf}(\partial_{\otimes 1}, 
(\partial+\lambda)_{\otimes 2}) q\otimes h - 2 A_{qh}(\partial_{\otimes 1}, 
(\partial+\lambda)_{\otimes 2}) q\otimes e) \\
+ (A_{qf}(\partial_{\otimes 2}, 
(\partial+\lambda)_{\otimes 1}) h \otimes q - 2 A_{qh}(\partial_{\otimes 2}, 
(\partial+\lambda)_{\otimes 1}) e \otimes q) \\
+ (A_{fl}((\partial+\lambda)_{\otimes 2}, 
\partial_{\otimes 1}) l\otimes h - 2 A_{hl}((\partial+\lambda)_{\otimes 2}, 
\partial_{\otimes 1}) l\otimes e).
\end{multline*}

Since $\Cur(\sl_2(\mathbb{C}))$ is a free $\mathbb{C}[\partial]$-module, 
we may omit $g(\partial^{\otimes 2})$, and all coefficients at $q\otimes l$ have to be zero.
For example, the coefficient at $h \otimes e$ equals
$$
0 = A_{fe}((\partial+\lambda)_{\otimes 1},\partial_{\otimes 2}) 
- 2 A_{hh}(\partial_{\otimes 1},(\partial+\lambda)_{\otimes 2}) 
+ A_{ef}(\partial_{\otimes 2},(\partial+\lambda)_{\otimes 1}) 
- 2 A_{hh}((\partial+\lambda)_{\otimes 2},\partial_{\otimes 1}).
$$

The equations are written modulo $\lambda = - \partial^{\otimes 2}$,
so, one can transfer polynomials over tensor product by a rule 
$f(\partial) \otimes 1 = 1 \otimes f(- \partial - \lambda)$.
Without loss of generality, gather all the polynomials at $h\otimes e$ on the first tensor factor:
\begin{multline*} 
0 = A_{fe}((\partial+\lambda)_{\otimes 1},(-\partial-\lambda)_{\otimes 1})
 - 2 A_{hh}(\partial_{\otimes 1},- \partial_{\otimes 1}) \\
 + A_{ef}((-\partial-\lambda)_{\otimes 1},(\partial+\lambda)_{\otimes 1}) 
 - 2 A_{hh}(- \partial_{\otimes 1},\partial_{\otimes 1}).
\end{multline*} 
Denoting $A_{ql}(x, - x) = A'_{ql}(x)$, we may rewrite the expression as follows,
$$
0 = A'_{fe}(\partial_{\otimes 1}+\lambda)
 - 2A'_{hh}(\partial_{\otimes 1}) 
 + A'_{ef}(-\partial_{\otimes 1}-\lambda) 
 - 2A'_{hh}(-\partial_{\otimes 1}).
$$
For $\partial = 0$, we get
$0 = A'_{fe}(\lambda) + A'_{ef}(-\lambda) - 4A'_{hh}(0)$.
By~\eqref{gen-sol}, $A'_{hh}(0) = \zeta$. Hence,
\begin{equation} \label{cond-fe-ef}
A'_{fe}(\lambda) + A'_{ef}(-\lambda) = 4 \zeta.
\end{equation}

Applying~\eqref{cond-fe-ef}, we obtain
$A'_{hh}(\lambda) + A'_{hh}(-\lambda) = 2 \zeta$.

Considering actions of $g(\partial)f$ and $g(\partial)h$ 
as well as other projections, the equation~\eqref{ConfInv} is equivalent to the following list of relations:
\begin{equation} \label{cond-full}
\begin{gathered}
A'_{ee}(\lambda) + A'_{ee}(-\lambda) = 0,\quad
A'_{fe}(\lambda) + A'_{ef}(-\lambda) = 4 \zeta,\\
A'_{he}(\lambda) + A'_{eh}(-\lambda) = 0,\quad
A'_{ff}(\lambda) + A'_{ff}(-\lambda) = 0, \\
A'_{hf}(\lambda) + A'_{fh}(-\lambda) = 0,\quad
A'_{hh}(\lambda) + A'_{hh}(-\lambda) = 2 \zeta.
\end{gathered}
\end{equation}

\subsection{Properties of a~solution to the (weak) CCYBE}

Let $r = \sum \nolimits A_{ql}(\partial_{\otimes 1}, 
\partial_{\otimes 2}) q\otimes l$ be 
a solution to the CCYBE on $\Cur(\sl_2(\mathbb{C}))$, it means that 
$r$~satisfies \eqref{weak-CCYBE} and \eqref{cond-full}. 
As above, we project the equalities~\eqref{weak-CCYBE} on basic tensors $q\otimes l\otimes m$, where $q,l,m\in\{e,f,h\}$. 
Applying sesquilinearity~\eqref{Conf:sesqui}, we get
\begin{multline} \label{rr}
\llbracket r, r\rrbracket 
 := \sum\limits_{q,q',l,l'=1}^3 (A_{ql}(- \lambda, \partial_{\otimes 2})
 A_{q'l'}(\partial_{\otimes 1} + \lambda, \partial_{\otimes 3})
 [q, q'] \otimes l \otimes l' \rvert_{\lambda = \partial_{\otimes 2}} \\
 - A_{ql}(\partial_{\otimes 1}, \partial_{\otimes 2} + \lambda) 
 A_{q'l'}(- \lambda, \partial_{\otimes 3})
 q \otimes [q', l] \otimes l' \rvert_{\lambda = \partial_{\otimes 3}} \\
 - A_{ql}(\partial_{\otimes 1}, \partial_{\otimes 3} + \lambda) 
 A_{q'l'}(\partial_{\otimes 2}, - \lambda)
 q \otimes q' \otimes [l', l] \rvert_{\lambda = \partial_{\otimes 2}}) \\
 = \sum\nolimits (A_{ql}(- \partial_{\otimes 2}, \partial_{\otimes 2})
 A_{q'l'}(\partial_{\otimes 1} + \partial_{\otimes 2}, \partial_{\otimes 3})
 [q, q'] \otimes l \otimes l' \\
 - A_{ql}(\partial_{\otimes 1}, \partial_{\otimes 2} + \partial_{\otimes 3}) 
 A_{q'l'}(- \partial_{\otimes 3}, \partial_{\otimes 3})
 q \otimes [q', l] \otimes l' \\
 - A_{ql}(\partial_{\otimes 1}, \partial_{\otimes 2} + \partial_{\otimes 3}) 
 A_{q'l'}(\partial_{\otimes 2}, - \partial_{\otimes 2})
 q \otimes q' \otimes [l', l]).
\end{multline}
Now, we compute $\llbracket r,r\rrbracket$ modulo $\partial^{\otimes3}$ expressing $\partial_{\otimes 1}$ via others, 
\begin{multline*}
\llbracket r,r\rrbracket
 = \sum\nolimits (A'_{ql}(-\partial_{\otimes 2})
 A'_{q'l'}(-\partial_{\otimes 3})[q, q'] \otimes l \otimes l' \\
 - A'_{ql}(-\partial_{\otimes 2} - \partial_{\otimes 3}) 
 A'_{q'l'}(- \partial_{\otimes 3})q \otimes [q', l] \otimes l' 
 - A'_{ql}(-\partial_{\otimes 2} - \partial_{\otimes 3}) 
 A'_{q'l'}(\partial_{\otimes 2})q \otimes q' \otimes [l', l]) = 0.
\end{multline*} 

One can directly check that modulo~\eqref{cond-full} the projections
of $\llbracket r,r\rrbracket$ on $q\otimes l\otimes m$ and 
on $\sigma(q)\otimes \sigma(l)\otimes \sigma(m)$
coincide for all $q,l,m\in\{e,f,h\}$ and $\sigma\in S_3$.
Let us show an example how it works. 
For $h \otimes f \otimes f$, we have
\begin{multline*}
0 = A'_{ef}(-\partial_{\otimes 2})A'_{ff}(-\partial_{\otimes 3})[e, f]\otimes f \otimes f 
+ A'_{ff}(-\partial_{\otimes 2})A'_{ef}(-\partial_{\otimes 3})[f, e]\otimes f \otimes f \\
- A'_{hf}(-\partial_{\otimes 2}-\partial_{\otimes 3}) 
A'_{hf}(-\partial_{\otimes 3})h \otimes [h, f] \otimes f 
+ A'_{hh}(-\partial_{\otimes 2}-\partial_{\otimes 3}) 
A'_{ff}(-\partial_{\otimes 3}) h \otimes [f, h] \otimes f \\
- A'_{hf}(-\partial_{\otimes 2}-\partial_{\otimes 3})
A'_{fh}(\partial_{\otimes 2}) h \otimes f \otimes [h, f] 
 + A'_{hh}(\partial_{-\otimes 2}-\partial_{\otimes 3}) 
A'_{ff}(\partial_{\otimes 2}) h \otimes f \otimes [f, h],
\end{multline*}
which gives a zero coefficient at $h \otimes f \otimes f$
\begin{multline*}
0 = A'_{ef}(-\partial_{\otimes 2})A'_{ff}(-\partial_{\otimes 3}) - A'_{ff}(-\partial_{\otimes 2})A'_{ef}(-\partial_{\otimes 3}) 
+ 2 A'_{hf}(-\partial_{\otimes 2}-\partial_{\otimes 3}) A'_{hf}(- \partial_{\otimes 3}) \\
- 2 A'_{hh}(-\partial_{\otimes 2}-\partial_{\otimes 3}) A'_{ff}(- \partial_{\otimes 3}) + 2 A'_{hf}(-\partial_{\otimes 2}-\partial_{\otimes 3}) A'_{fh}(\partial_{\otimes 2}) - 2 A'_{hh}(-\partial_{\otimes 2}-\partial_{\otimes 3})A'_{ff}(\partial_{\otimes 2}),
\end{multline*}
At $f \otimes h \otimes f$, we analogously get the coefficient
\begin{multline*}
0 = 2 A'_{fh}(- \partial_{\otimes 2})A'_{hf}(- \partial_{\otimes 3}) - 
2 A'_{hh}(- \partial_{\otimes 2} )A'_{ff}(- \partial_{\otimes 3}) 
+ A'_{fe}(- \partial_{\otimes 2} - \partial_{\otimes 3})A'_{ff}(- \partial_{\otimes 3})  \\
- A'_{ff}(- \partial_{\otimes 2} - \partial_{\otimes 3}) 
A'_{ef}(- \partial_{\otimes 3})  
+ 2 A'_{ff}(- \partial_{\otimes 2} - \partial_{\otimes 3}) 
A'_{hh}(\partial_{\otimes 2})  
- 2 A'_{fh}(- \partial_{\otimes 2} - \partial_{\otimes 3})
A'_{hf}(\partial_{\otimes 2}).
\end{multline*}
Substituting $\partial_{\otimes 2}$ by $- \partial_{\otimes 2} - \partial_{\otimes 3}$, we transform with help of~\eqref{cond-full} the coefficient at $f\otimes h\otimes f$ to the one at $h\otimes f\otimes f$. 

Therefore, we have ten meaningful projections listed below
at $e^{\otimes 3},h\otimes e^{\otimes2},f^{\otimes 3},h\otimes f^{\otimes2},
e\otimes h^{\otimes2},f\otimes h^{\otimes2},f\otimes e^{\otimes2},
e\otimes f^{\otimes2},h^{\otimes 3},e\otimes f\otimes h$ respectively. 
For brevity, we denote $\partial_{\otimes 2}$ and $\partial_{\otimes 3}$ by $x$ and $y$:
\begin{multline} \label{eee}
0 = A'_{he}(- x)A'_{ee}(- y) 
- A'_{ee}(- x )A'_{he}(- y) 
+ A'_{eh}(- x - y) A'_{ee}(- y) \\
- A'_{ee}(- x - y) A'_{he}(- y) 
+ A'_{eh}(- x - y) A'_{ee}(x) 
- A'_{ee}(- x - y) A'_{eh}(x),
\end{multline} 
\begin{multline} \label{hee}
0 = A'_{ee}(- x)A'_{fe}(- y) 
- A'_{fe}(- x )A'_{ee}(- y) 
+ 2 A'_{hh}(- x - y) A'_{ee}(- y)  \\
- 2 A'_{he}(- x - y) A'_{he}(- y)  
+ 2 A'_{hh}(- x - y) A'_{ee}(x)  
- 2 A'_{he}(- x - y) A'_{eh}(x),
\end{multline} 
\begin{multline} \label{fff}
0 = A'_{ff}(- x)A'_{hf}(- y) 
- A'_{hf}(- x )A'_{ff}(- y) 
+ A'_{ff}(- x - y) A'_{hf}(- y) \\
- A'_{fh}(- x - y) A'_{ff}(- y)  
+ A'_{ff}(- x - y) A'_{fh}(x)  
- A'_{fh}(- x - y)A'_{ff}(x),
\end{multline} 
\begin{multline} \label{hff}
0 = A'_{ef}(- x)A'_{ff}(- y) 
- A'_{ff}(- x )A'_{ef}(- y) 
+ 2 A'_{hf}(- x - y) A'_{hf}(- y)  \\
- 2 A'_{hh}(- x - y) A'_{ff}(- y)  
+ 2 A'_{hf}(- x - y) A'_{fh}(x)  
- 2 A'_{hh}(- x - y)A'_{ff}(x),
\end{multline} 
\begin{multline} \label{ehh}
0 = 2 A'_{hh}(- x)A'_{eh}(- y) 
- 2 A'_{eh}(- x )A'_{hh}(- y) 
+ A'_{ee}(- x - y) A'_{fh}(- y)  \\
- A'_{ef}(- x - y) A'_{eh}(- y)  
+  A'_{ee}(- x - y) A'_{hf}(x)  
- A'_{ef}(- x - y) A'_{he}(x),
\end{multline} 
\begin{multline} \label{fhh}
0 = 2 A'_{fh}(- x)A'_{hh}(- y) 
- 2 A'_{hh}(- x )A'_{fh}(- y) 
+ A'_{fe}(- x - y) A'_{fh}(- y)  \\
- A'_{ff}(- x - y) A'_{eh}(- y)  
+ A'_{fe}(- x - y) A'_{hf}(x)  
- A'_{ff}(- x - y) A'_{he}(x),
\end{multline} 
\begin{multline} \label{fee}
0 = A'_{fe}(- x)A'_{he}(- y) 
- A'_{he}(- x )A'_{fe}(- y) 
+ A'_{fh}(- x - y) A'_{ee}(- y)  \\
- A'_{fe}(- x - y) A'_{he}(- y)  
+ A'_{fh}(- x - y) A'_{ee}(x)  
- A'_{fe}(- x - y)A'_{eh}(x),
\end{multline} 
\begin{multline} \label{eff}
0 = A'_{hf}(- x)A'_{ef}(- y) 
- A'_{ef}(- x )A'_{hf}(- y) 
+ A'_{ef}(- x - y) A'_{hf}(- y) \\
- A'_{eh}(- x - y) A'_{ff}(- y)  
+ A'_{ef}(- x - y) A'_{fh}(x)  
- A'_{eh}(- x - y) A'_{ff}(x),
\end{multline} 
\begin{multline} \label{hhh}
0 = A'_{eh}(- x)A'_{fh}(- y) 
- A'_{fh}(- x )A'_{eh}(- y) 
+ A'_{he}(- x - y) A'_{fh}(- y)  \\
- A'_{hf}(- x - y) A'_{eh}(- y)  
+  A'_{he}(- x - y) A'_{hf}(x)  
- A'_{hf}(- x - y) A'_{he}(x),
\end{multline} 
\begin{multline} \label{efh}
0 = 2 A'_{hf}(- x)A'_{eh}(- y) 
- 2 A'_{ef}(- x )A'_{hh}(- y) 
+ 2 A'_{ef}(- x - y) A'_{hh}(- y) \\
- 2 A'_{eh}(- x - y) A'_{fh}(- y)  
+ A'_{ee}(- x - y) A'_{ff}(x)  
- A'_{ef}(- x - y) A'_{fe}(x).
\end{multline} 

For the case of the weak CCYBE we have very close system of defining equations. 
More detailed, the solution to the weak CCYBE is defined by \eqref{eee}--\eqref{hhh} and a set of variations of~\eqref{efh}.

To show this, we consider projections of $(a_\mu \llbracket r,r\rrbracket)|_{b\otimes c\otimes d}$,
where $a,b,c,d\in\{e,f,h\}$ and at least two from $b,c,d$ do not have the form $[a,x]$ for some $x\in \sl_2(\mathbb{C})$.
Hence, we get the equations~\eqref{eee}--\eqref{fhh}.
For example, to obtain \eqref{fff}, one may consider the projection of~$e_\mu \llbracket r,r\rrbracket$ on $h \otimes f \otimes f$.
The only way to get the projection is to act by $e$ on $f\otimes f \otimes f$,
\begin{multline*}
e_\mu ( (2 A_{ff}(-\partial_{\otimes 2}, \partial_{\otimes 2})
A_{hf}(\partial_{\otimes 1} + \partial_{\otimes 2}, \partial_{\otimes 3}) 
- 2 A_{hf}(-\partial_{\otimes 2}, \partial_{\otimes 2}) A_{ff}(\partial_{\otimes 1} + 
\partial_{\otimes 2}, \partial_{\otimes 3})
\\
+ 2 A_{ff}(\partial_{\otimes 1}, \partial_{\otimes 2} + \partial_{\otimes 3}) 
A_{hf}(- \partial_{\otimes 3}, \partial_{\otimes 3})
- 2 A_{fh}(\partial_{\otimes 1}, \partial_{\otimes 2} + \partial_{\otimes 3}) 
A_{ff}(- \partial_{\otimes 3}, \partial_{\otimes 3})
\\
+ 2 A_{ff}(\partial_{\otimes 1}, \partial_{\otimes 2} + \partial_{\otimes 3})
A_{fh}(\partial_{\otimes 2}, -\partial_{\otimes 2}) - 
2 A_{fh}(\partial_{\otimes 1}, \partial_{\otimes 2} + \partial_{\otimes 3})
A_{ff}(\partial_{\otimes 2}, -\partial_{\otimes 2})) f \otimes f \otimes f ).
\end{multline*}
Applying $e$ to the first tensor factor, we get the projection
on $h \otimes f \otimes f$,
\begin{multline*}
 (2 A_{ff}(-\partial_{\otimes 2}, \partial_{\otimes 2})
A_{hf}(\partial_{\otimes 1} + \mu + \partial_{\otimes 2}, \partial_{\otimes 3}) 
- 2 A_{hf}(-\partial_{\otimes 2}, \partial_{\otimes 2}) A_{ff}(\partial_{\otimes 1} + \mu + 
\partial_{\otimes 2}, \partial_{\otimes 3})
\\
+ 2 A_{ff}(\partial_{\otimes 1} + \mu, \partial_{\otimes 2} + \partial_{\otimes 3}) 
A_{hf}(- \partial_{\otimes 3}, \partial_{\otimes 3})
- 2 A_{fh}(\partial_{\otimes 1} + \mu, \partial_{\otimes 2} + \partial_{\otimes 3}) 
A_{ff}(- \partial_{\otimes 3}, \partial_{\otimes 3})
\\
+ 2 A_{ff}(\partial_{\otimes 1} + \mu, \partial_{\otimes 2} + \partial_{\otimes 3})
A_{fh}(\partial_{\otimes 2}, -\partial_{\otimes 2}) - 
2 A_{fh}(\partial_{\otimes 1} + \mu, \partial_{\otimes 2} + \partial_{\otimes 3})
A_{ff}(\partial_{\otimes 2}, -\partial_{\otimes 2})).
\end{multline*}
which equals to~\eqref{fff} modulo $\mu = -\partial^{\otimes 3}$.

Further, we write down projections of $(a_\mu \llbracket r,r\rrbracket)|_{b\otimes c\otimes d}$,
where $a,b,c,d\in\{e,f,h\}$ and at least one from $b,c,d$ does not have the form $[a,x]$ for some $x\in \sl_2(\mathbb{C})$.
With the help of \eqref{eee}--\eqref{fhh}, way obtain the equations \eqref{fee} and \eqref{eff}.
For instance, consider the projection of $e_\mu \llbracket r,r\rrbracket$ on 
$e \otimes h \otimes f$. Computing, two groups of six summands appear: obtained from the projections of~$\llbracket r,r\rrbracket$ on 
$h \otimes h \otimes f$ and $e \otimes f \otimes f$ respectively. 
Since the first group up to change of variables equals~\eqref{fhh},
we obtain~\eqref{eff}. 

Now, we study the projections of $(a_\mu \llbracket r,r\rrbracket)|_{b\otimes c\otimes d}$, where $\{b,c,d\} = \{e,f,h\}$.
For example, 
\begin{multline}\label{efh'}
0 
 = (h_\mu \llbracket r,r\rrbracket)|_{e\otimes f\otimes h}
 = 2 A'_{hf}(- x)A'_{eh}(- y) 
 - 2 A'_{ef}(- x)A'_{hh}(- y) 
 + 2 A'_{ef}(- x - y) A'_{hh}(- y) \\
 - 2 A'_{eh}(- x - y) A'_{fh}(- y)  
 + A'_{ee}(- x - y) A'_{ff}(x)  
 - A'_{ef}(- x - y) A'_{fe}(x) \\
 - 2 A'_{hf}(y + z)A'_{eh}(- y) 
 + 2 A'_{ef}(y + z)A'_{hh}(- y) 
 - 2 A'_{ef}(z) A'_{hh}(- y)  \\
 + 2 A'_{eh}(z) A'_{fh}(- y) 
 - A'_{ee}(z) A'_{ff}(- y - z) 
 + A'_{ef}(z) A'_{fe}(- y - z);
\end{multline}
\begin{multline}\label{efh''}
0 = (e_\mu \llbracket r,r\rrbracket)|_{e\otimes f\otimes e}
 = - A'_{ef}(-x)A'_{fe}(- y) 
 + A'_{ff}(-x)A'_{ee}(- y) 
 + 2 A'_{hh}(-x-y) A'_{fe}(- y) \\
 - 2 A'_{hf}(-x-y) A'_{he}(- y) 
 + 2A'_{he}(-x-y) A'_{fh}(x) 
 - 2A'_{hh}(-x-y) A'_{fe}(x) \\
 + 2 A'_{ef}(- x)A'_{hh}(x+z) 
 - 2 A'_{hf}(- x)A'_{eh}(x+z) 
 - 2 A'_{ef}(z) A'_{hh}(x+z) \\
 + 2 A'_{eh}(z) A'_{fh}(x+z) 
 - A'_{ee}(z) A'_{ff}(x)
 + A'_{ef}(z) A'_{fe}(x).
\end{multline}
As above, we denote $\partial_{\otimes 2}$ and $\partial_{\otimes 3}$ by $x$ and $y$,
and $z$ stands for $\partial_{\otimes 1}$.

There are several nonequivalent (in terms of change of variables) equations 
similar to~\eqref{efh'} and~\eqref{efh''}. 
It turns out that we do not need the whole system of these equations.

\begin{remark} \label{remarkOn(22)}
Let us take the equality \eqref{efh} and consider it as the condition on a solution to the weak CCYBE. 
By Lemma~\ref{lem:ZeroLevel} we know its value when $x = y = 0$, hence we write down the new ``artificial'' equality
\begin{multline} \label{efh-shift}
0 = 2 A'_{hf}(- x)A'_{eh}(- y) 
- 2 A'_{ef}(- x )A'_{hh}(- y) 
+ 2 A'_{ef}(- x - y) A'_{hh}(- y) \\
- 2 A'_{eh}(- x - y) A'_{fh}(- y) 
+ A'_{ee}(- x - y) A'_{ff}(x)  \\
- A'_{ef}(- x - y) A'_{fe}(x) + 
4 \alpha \gamma + (4\zeta - \beta)\beta.
\end{multline} 
Up to sign, \eqref{efh-shift}$\rvert_{x = 0}$ coincides with~\eqref{efh''}$\rvert_{x=y=0}$ under the change $y\to -z$ and \eqref{efh-shift}$\rvert_{y = 0}$ coincides with~\eqref{efh'}$\rvert_{x=y=0}$ under the change $x\to -z$.
Moreover, we act by the automorphism~$\Psi$ of $\Cur(\sl_2{\mathbb{(C)}})$ defined as follows,
\begin{equation} \label{PsiAut}
\Psi(e) = f,\quad
\Psi(f) = e,\quad 
\Psi(h) = -h, 
\end{equation}
on~\eqref{efh-shift} and then take $x+y = 0$.
Up to sign, the obtained expression coincides with~\eqref{efh''}$\rvert_{x=y=0}$ under the change $-y\to z$ and the application of~\eqref{cond-full}.
\end{remark}

As for \eqref{hhh}, under action of $e$, representatives of the pattern 
$e \otimes h \otimes h$ arise from the projections 
of~$\llbracket r,r\rrbracket$ on $h \otimes h \otimes h$, 
$e \otimes f \otimes h$, and $e \otimes h \otimes f$. The resulting 
equation is as follows:
\begin{multline*}
-2(A'_{eh}(- x)A'_{fh}(- y) 
- A'_{fh}(- x )A'_{eh}(- y) 
+ A'_{he}(- x - y) A'_{fh}(- y) \\
- A'_{hf}(- x - y) A'_{eh}(- y)  
+  A'_{he}(- x - y) A'_{hf}(x)  
- A'_{hf}(- x - y) A'_{he}(x)) \\
+ 2 A'_{hf}(y+z)A'_{eh}(- y) 
- 2 A'_{ef}(y+z)A'_{hh}(- y) 
+ 2 A'_{ef}(z) A'_{hh}(- y) \\
- 2 A'_{eh}(z) A'_{fh}(- y) 
+ A'_{ee}(z) A'_{ff}(-y-z)  
- A'_{ef}(z) A'_{fe}(-y-z)  \\
+ 2 A'_{hh}(- x)A'_{ef}(x+z) 
- 2 A'_{eh}(- x)A'_{hf}(x+z) 
+ A'_{ee}(z) A'_{ff}(x+z) \\
- A'_{ef}(z) A'_{ef}(x+z) 
+ 2 A'_{ef}(z) A'_{hh}(x)  
- 2 A'_{eh}(z) A'_{hf}(x)
= 0,
\end{multline*}
that is exactly \eqref{hhh} setting $z = 0$, since the sum of the last 12 terms is zero by the projection of~\eqref{efh-shift} at $x + y = 0$.

Now, we prove some kind of symmetry for the polynomials $A'_{ql}$.
Before this, we write down the values of $A'_{ql}(0)$ by Lemma~\ref{lem:ZeroLevel}:
\begin{equation} \label{eq:A'ql(0)}
\begin{gathered}
A'_{ee}(0) = A'_{ff}(0) = 0, \quad
A'_{ef}(0) = 4\zeta-\beta,\quad
A'_{fe}(0) = \beta, \quad
A'_{he}(0) = -A'_{eh}(0) = \alpha, \\
A'_{hf}(0) = -A'_{fh}(0) = \gamma, \quad
A'_{hh}(0) = \zeta.
\end{gathered}
\end{equation}

\begin{lemma}\label{symm}
For all $q, l \in \{e, f, h\}$, we have
$A'_{ql}(x) - A'_{ql}(0) = A'_{lq}(x) - A'_{lq}(0)$.
\end{lemma}
\begin{proof}
The statement is trivial for $q = l$.
Let us take~\eqref{eee} and \eqref{hee} with $x = 0$: 
\begin{gather} 
0 = A'_{ee}(-y) (A'_{eh}(-y) - A'_{he}(-y) + 2 \alpha), \label{eee_0}\\
0 = A'_{ee}(-y) (A'_{hh}(-y) - \beta/2) - A'_{he}(-y) (A'_{he}(-y) - \alpha). \label{hee_0}
\end{gather}
If $A'_{ee}(-y) \ne 0$, then \eqref{eee_0} implies Lemma for 
the pair $A'_{eh}(x)$, $A'_{he}(x)$. 
Otherwise, $A'_{he}(-y) (A'_{he}(-y) - \alpha) = 0$ by~\eqref{hee_0}.
Since $A'_{he}(0) = \alpha$, we derive that $A'_{he}(-y) = \alpha$.
By~\eqref{cond-full}, we are done for the case $\{q,l\} = \{e,h\}$.

The equalities \eqref{fff} and \eqref{hff} for $x = 0$ equal
\begin{gather} 
0 = A'_{ff}(-y) (A'_{hf}(-y) - A'_{fh}(-y) -  2\gamma), \label{fff_0} \\
0 = A'_{ff}(-y) (A'_{hh}(-y) - 2 \zeta + \beta/2) 
- A'_{hf}(-y) (A'_{hf}(-y) - \gamma). \label{hff_0}
\end{gather}
We analogously obtain the statement for the case $\{q,l\} = \{f,h\}$.

Finally, we prove that $A'_{ef}(x) - A'_{ef}(0) = A'_{fe}(x) - A'_{fe}(0)$. For this, consider 
\eqref{fee} and \eqref{ehh} for $x = 0$ applying \eqref{cond-full},
\begin{gather}
0 =  A'_{fh}(- y) A'_{ee}(- y)  
+ (A'_{ef}(y) - 4\zeta + \beta) A'_{he}(- y), \label{a} \\
0 = A'_{ee}(- y) (A'_{fh}(- y) + \gamma)  
- (A'_{ef}(- y) - 2 \zeta) (A'_{eh}(- y) + \alpha)
+ 2 \alpha (A'_{hh}(- y) - \zeta), \label{b}
\end{gather}
Put $-y$ instead of~$y$ in~\eqref{a},~\eqref{b} to get by \eqref{cond-full} and by the already proved cases of Lemma (below we leave $A'_{ef}$ not applying~\eqref{cond-full}):
\begin{gather}
0 =  (A'_{fh}(- y) + 2\gamma) A'_{ee}(- y) 
 - (A'_{ef}(- y) - 4\zeta + \beta) (A'_{he}(- y) - 2\alpha), \label{a'} \\
0 = A'_{ee}(- y) (A'_{fh}(- y) + \gamma)  
 + (A'_{ef}(y) - 2\zeta) (A'_{eh}(- y) + \alpha)
 - 2 \alpha (A'_{hh}(- y) {-} \zeta).  \label{b'}
\end{gather} 
Then add \eqref{a} to \eqref{a'} and \eqref{b} to \eqref{b'}:
\begin{multline} \label{a''}
0 = 2 (A'_{fh}(- y) + \gamma) A'_{ee}(- y)\\
- (A'_{ef}(- y) - A'_{ef}(y)) (A'_{he}(- y) - \alpha) 
+ \alpha (A'_{ef}(- y) + A'_{ef}(y) - 8\zeta + 2\beta),  
\end{multline}
\begin{equation} \label{b''}
0 = 2 A'_{ee}(- y) (A'_{fh}(- y) + \gamma)  
+ (A'_{ef}(y) - A'_{ef}(- y)) (A'_{eh}(- y) + \alpha).
\end{equation} 

Subtracting~\eqref{a''} from~\eqref{b''}, we get
$$
\alpha(A'_{ef}(- y) + A'_{ef}(y) - 8\zeta + 2\beta) = 0.
$$
If $\alpha\neq0$, we finish the proof by~\eqref{cond-full}. 
If $\alpha = 0$, then subtract \eqref{b} from \eqref{b'}
to obtain 
$$
A'_{eh}(-y)(A'_{ef}(- y) + A'_{ef}(y) - 4\zeta) = 0. 
$$
If $A'_{eh}\not\equiv0$, then we get the equality
$A'_{ef}(y) = A'_{fe}(y)$ by~\eqref{cond-full} and, hence, 
$A'_{ef}(0) = A'_{fe}(0)$, i.\,e., $\beta = 2\zeta$. 
Thus, 
$$
A'_{ef}(y) - A'_{ef}(0) 
 = A'_{ef}(y) - \beta 
 = A'_{fe}(y) - \beta
 = A'_{fe}(y) - A'_{fe}(0),
$$
as required.
If $A'_{eh}\equiv0$, then we take~\eqref{efh-shift} with $y = 0$ and get
\begin{equation}\label{lalala}
A'_{ee}(-x)A'_{ff}(x) - A'_{ef}(-x)A'_{fe}(x) + \beta(4\zeta-\beta) = 0.
\end{equation}
This relation considered for $x$ and $-x$ by~\eqref{cond-full} implies
$$
(A'_{ef}(x)-A'_{ef}(-x))(A'_{ef}(- y) + A'_{ef}(y) - 4\zeta) = 0.
$$
Above we have shown that $A'_{ef}(- y) + A'_{ef}(y) - 4\zeta = 0$ gives the required equality. 

Suppose that $A'_{ef}(x) = A'_{ef}(-x)$.
Let us show that $A'_{fe}(x)$ is constant, it is sufficient to finish the proof. 
Consider \eqref{hee} for $x + y = 0$. 
Applying \eqref{cond-full}, we obtain
$$
0 = A'_{ee}(x) (A'_{ef}(x) - 2\zeta).
$$
If $A'_{ee}(x) \not\equiv 0$, then $A'_{ef}(x) = 2\zeta$.
If $A'_{ee}(x)\equiv 0$, then $A'_{ef}(x)$ is constant by~\eqref{lalala}.
\end{proof}

\begin{corollary}\label{odd}
We have 
$A'_{ql}(- x) - A'_{ql}(0) = -(A'_{ql}(x) - A'_{ql}(0))$
for all $q, l \in \{e, f, h\}$. 
In other words, $A'_{ql}(x) - A'_{ql}(0)$ is an odd polynomial.
\end{corollary}
\begin{proof}
Combine Lemma~\ref{symm} and the equations~\eqref{cond-full}.
\end{proof}

In the light of Lemma~\ref{symm} and Corollary~\ref{odd}, let us introduce the following
presentation of the polynomials for $q, l \in \{e, f, h\}$:
$$
A'_{ql}(x) - A'_{ql}(0) = a_{ql} x f_{ql}(x^2),
$$
where $f_{ql}(x^2)$ is a unitary polynomial.
Writing scalars $a_{ql}$ in the matrix $M\in M_3(\mathbb{C})$ due to the order $e,f,h$, we get a~symmetric one.

\begin{lemma} \label{Mff}
The following equalities hold:
\begin{gather} 
a_{he} A'_{ee}(x) = a_{ee} (A'_{he}(x) - \alpha), \quad
a_{he} (A'_{hh}(x) - \zeta) = a_{hh} (A'_{he}(x) - \alpha), \label{ee=he} \\
a_{ee} a_{hh} = a_{he}^2, \quad 
(2\zeta - \beta) a_{ee} = 2 \alpha a_{he}, \quad 
2\alpha a_{hh} = (2 \zeta - \beta) a_{he}. \label{ee=hecoeff}
\end{gather}
\end{lemma}

\begin{proof}
The first equality of~\eqref{ee=hecoeff} comes from the even part of~\eqref{hee_0},
which equals 
$$
0 = y^2(a_{ee}a_{hh}f_{ee}(y^2)f_{hh}(y^2) - a_{he}^2f_{he}(y^2)).
$$
From~\eqref{hee_0}, we have by Corollary~\ref{odd},
\begin{equation}\label{eq:Lemma2eehh}
0 = A'_{ee}(y) (A'_{hh}(-y) - \beta/2) - A'_{he}(-y) (A'_{he}(y) - \alpha).
\end{equation}
Applying~\eqref{hee_0} with $-y$ and Corollary~\ref{odd}, we deduce
\begin{multline}\label{eq:Lemma2Long}
0 = A'_{ee}(y) (A'_{hh}(y) - \beta/2) 
 - A'_{he}(y)(A'_{he}(y) - \alpha) \\
 = A'_{ee}(y) (- A'_{hh}(-y) + 2 \zeta - \beta/2) 
 + (A'_{he}(-y) - 2\alpha) (A'_{he}(y) - \alpha). 
\end{multline}

The sum of~\eqref{eq:Lemma2eehh} and~\eqref{eq:Lemma2Long} gives us 
\begin{equation}\label{Aee=Ahe}
(2 \zeta - \beta)A'_{ee}(y) = 2 \alpha (A'_{he}(y) - \alpha).
\end{equation}
Hence, we prove the second equality from~\eqref{ee=hecoeff}.
It remains to check~\eqref{ee=he} and the last equality of~\eqref{ee=hecoeff}.

Application of the first two equations in \eqref{ee=hecoeff} gives 
$(2\zeta - \beta) a_{he}^2 = (2\zeta - \beta) a_{ee} a_{hh} = 2 \alpha a_{he}
a_{hh}$, hence, $a_{he}((2\zeta - \beta) a_{he} - 2 \alpha a_{hh}) = 0$. Then
either we are done with~\eqref{ee=hecoeff} or $a_{he} = 0$. 
In the latter case,
it remains to study the case when $a_{ee} = 0$. 

Suppose that $a_{he} = 0$, equivalently, $A'_{he}(y)-\alpha\equiv0$.
Then~\eqref{ee=he} is fulfilled.
Consider~\eqref{ehh} at $y = 0$:
$$
0 = 2\alpha(\zeta - A'_{hh}(-x)) + A'_{ee}(-x)(A'_{hf}(x)-\gamma) 
 = 2\alpha a_{hh}xf_{hh}(x^2) - a_{ee}a_{hf}x^2f_{ee}(x^2)f_{hf}(x^2).
$$
Now, we finish the proof of the third equality of~\eqref{ee=hecoeff}, since the remaining case $a_{he} = a_{ee} = 0$ implies $\alpha a_{hh} = 0$.

If $A'_{ee}(x)\equiv 0$, then by~\eqref{hee_0} we get 
$A'_{he}(-y)-\alpha\equiv0$, the already studied case. 

Now, we may suppose that $A'_{ee}(x)\not\equiv 0$ and $A'_{he}(y)-\alpha\not\equiv0$.
By~\eqref{Aee=Ahe}, either $\alpha = 2 \zeta - \beta = 0$ or 
$\alpha$, $2 \zeta - \beta \neq0$.

If $\alpha$, $2 \zeta - \beta \neq0$, then 
the equality~\eqref{Aee=Ahe} and its substitution into~\eqref{hee_0}
give us by~\eqref{ee=hecoeff} both equalities~\eqref{ee=he}.

Suppose that $\alpha = 2 \zeta - \beta = 0$.
Assume that there exists a nonzero root of~$A'_{ee}(y)$, say $\tau$.
Since $A'_{ee}(y) = a_{ee}yf_{ee}(y^2)$, $-\tau$ is also a~root of~$A'_{ee}(y)$. 
By~\eqref{hee_0}, $\pm\tau$ are roots of $A'_{he}(y)$. 
Let us choose $\tau$ as a root of $A'_{ee}(y)$ of maximal modulus.
Now, consider~\eqref{eee} for $x = \tau$:
$$
0 = A'_{ee}(- y)A'_{he}(- \tau - y)
  - A'_{ee}(- \tau - y)A'_{he}(- y),
$$
which can be rewritten as 
$$
g(y) = \frac{A'_{he}(y)}{A'_{ee}(y)} = \frac{A'_{he}(\tau + y)}{A'_{ee}(\tau + y)}.
$$
If 0~is a~pole of $g(y)$, then so is $\tau$. But if $\tau$ is a pole of $g(y)$, 
then so is $2\tau$, a contradiction with the choice of $\tau$.
If $(0\neq)\kappa$ is a~pole of~$f(y)$, then $|\tau+\kappa|>|\tau|$ or $|\tau-\kappa|>|\tau|$, which leads us again to a~contradiction. 
So, $g(y)$ is a~polynomial, it means that $A'_{he}(y) = \delta A'_{ee}(y)$ for some $\delta\in\mathbb{C}$. By the form of $A'_{ql}(y)$, we have $\delta a_{ee} = a_{he}$, it proves the first equality of~\eqref{ee=he} by~\eqref{ee=hecoeff}. 
The substitution of $A'_{he}(y) = \delta A'_{ee}(y)$ into~\eqref{hee_0}
give us the second equality from~\eqref{ee=he}.

Now, we consider the case when $A'_{ee}(y)$ has no roots except~0, it means that 
$A'_{ee}(y) = a_{ee} y^k$, $k$ is odd.
Suppose that there exists a nonzero root of~$A'_{he}(y)$, say $\tau$. 
By~\eqref{hee_0}, $\pm\tau$ are roots of $A'_{hh}(y) - \zeta$ as well. 
We write down~\eqref{hee} for $x+y = - \tau$, 
$$
0 = A'_{ee}(y+\tau) (A'_{ef}(y) - \beta) - A'_{ee}(y) (A'_{ef}(y+\tau) - \beta).
$$
In the same manner as above, we derive that $A'_{ef}(y) - \beta = a_{ef} y^k$.
Hereby, \eqref{ehh} induces for $x = 0$ the following relations
$$
a_{ee} y^k (A'_{fh}(y) +\gamma)
 =  A'_{ee}(y) (A'_{fh}(y) +\gamma)
 = (A'_{ef}(y) - \beta) A'_{eh}(y) 
 = a_{ef} y^k A'_{eh}(y).
$$
So, substitute the old relations and the new one $A'_{fh}(y) = (a_{ef}/a_{ee})A'_{eh}(y) - \gamma$ in~\eqref{ehh} to obtain
$$
(A'_{hh}(-x) - \zeta) A'_{eh}(-y) = A'_{eh}(-x) (A'_{hh}(-y) - \zeta).
$$
We may derive that $A'_{hh}(y)-\zeta = (a_{hh}/a_{eh})A'_{eh}(y)$, the second
equality of~\eqref{ee=he}, whereas substituting it into \eqref{eq:Lemma2eehh}
give us the rest.

The last case is when $A'_{ee}(y) = a_{ee}y^k$, $A'_{he}(y) = a_{he} y^n$, and $k$, $n$ are odd.
If $k = n$, then the first equality of~\eqref{ee=he} follows. Substituting it into~\eqref{eq:Lemma2eehh}, we get the second equality of~\eqref{ee=he}.
Suppose that $k > n$. Evaluate 
the $n$th partial derivation $\frac{\partial^n}{\partial (-x)^n}$ of \eqref{eee} at $x = 0$:
$$
0 = n!a_{he}a_{ee}y^k\left(-3+\binom{k}{n}\right). 
$$
Hence, $\binom{k}{n} = 3$, which for odd integers $k,n$ implies $k = 3$, $n = 1$. 
We arrive at a~contradiction with~\eqref{eq:Lemma2eehh}.

If $k<n$, then we analogously evaluate 
the $k$th partial derivation $\frac{\partial^k}{\partial (-y)^k}$ of \eqref{eee} at $y = 0$ to get $\binom{n}{k} = 3$, thus, similarly, $n = 3$, $k = 1$.
Therefore, \eqref{eq:Lemma2eehh} gives us $A'_{hh}(y) = a_{hh} y^5 +\beta/2$, and we get by~\eqref{hee}
\begin{multline} \label{finishHim}
0 = - a_{ee}x A'_{fe}(- y) 
+ A'_{fe}(- x) a_{ee}y 
+ 2 a_{hh}a_{ee}(x + y)^5y - \beta a_{ee}y \\
- 2 a_{he}^2(x + y)^3y^3  
- 2 a_{hh}a_{ee}(x + y)^5 x +\beta a_{ee}x 
+ 2 a_{he}^2(x + y)^3 x^3.
\end{multline} 
The coefficient at $x^2y^4$ in~\eqref{finishHim} equals 
$10a_{hh}a_{ee} - 6a_{he}^2 = 4a_{hh}a_{ee}\neq0$, a~contradiction.
\end{proof}

Within quite the same approach the analogous assertion can be proven.
The equalities listed below may be also derived with the help of the automorphism $\Psi$~\eqref{PsiAut} of $\Cur(\sl_2(\mathbb{C}))$.
For this, we apply the relation
$2\zeta - \beta = (A'_{ef}(0)-A'_{fe}(0))/2$.

\begin{lemma} \label{Mee}
The following equalities hold:
\begin{gather} 
a_{hf} A'_{ff}(x) = a_{ff} (A'_{hf}(x) - \gamma), \quad
a_{hf} (A'_{hh}(x) - \zeta) = a_{hh} (A'_{hf}(x) - \gamma), \label{ff=hf} \\
a_{ff} a_{hh} = a_{hf}^2, \quad (2\zeta - \beta) a_{ff} =
- 2 \gamma a_{hf}, \quad - 2 \gamma a_{hh} = (2 \zeta - \beta) a_{hf}. \label{ff=hfcoeff}
\end{gather}
\end{lemma}

Before stating the main result, we need one lemma.

\begin{lemma} \label{Mhh}
The following equalities hold:
\begin{gather} 
a_{fe} A'_{ee}(x) = a_{ee} (A'_{fe}(x) - \beta),
\quad a_{fe} A'_{ff}(x) = a_{ff} (A'_{fe}(x) - \beta),\label{ee=fe} \\
a_{ee} a_{ff} = a_{fe}^2, \quad \gamma a_{ee} = 
- \alpha a_{fe}, \quad \alpha a_{ff} = - \gamma a_{fe}.\label{ee=fecoeff}
\end{gather}
\end{lemma}

\begin{proof}
Consider \eqref{hee} and apply \eqref{ee=he} to the last two terms:
\begin{multline}
2(A'_{hh}(- x - y) A'_{ee}(x)  
 - A'_{he}(- x - y) A'_{eh}(x)) \\ 
 = 2 (-a_{hh}(x + y)f_{hh}((x+y)^2) + \zeta)
 a_{ee}xf_{ee}(x^2) \\
 - 2(-a_{he}(x + y)f_{he}((x+y)^2) + \alpha) (a_{eh}xf_{eh}(x^2) - \alpha) \\
 = 2(- a_{hh} a_{ee} (x+y)xf_{hh}((x+y)^2)f_{ee}(x^2) + a_{he}^2 (x+y)xf_{eh}((x+y)^2)f_{eh}(x^2)) \\
 - 2\alpha a_{he}(x+y)f_{eh}((x+y)^2) 
  + 2\zeta a_{ee}xf_{ee}(x^2) - 2\alpha a_{he} xf_{eh}(x^2) +2\alpha^2.
\end{multline}
Let us show that the sum of the first two terms is zero.
If $a_{eh}\neq0$, then by~\eqref{ee=he} we have $f_{ee} = f_{eh} = f_{hh}$ and the sum equals zero due to~\eqref{ee=hecoeff}. 
If $a_{eh} = 0$, then $a_{ee}a_{hh} = 0$ by~\eqref{ee=hecoeff} and both summands are zero. 

Analogously, the middle terms of~\eqref{hee} give us $$
2\alpha a_{he}(x+y)f_{eh}((x+y)^2) - 2\zeta a_{ee}yf_{ee}(y^2) + 2\alpha a_{he}yf_{eh}(y^2) - 2\alpha^2.
$$
We want to derive the equality
\begin{equation} \label{ee-feKey}
0 = A'_{ee}(- x)(A'_{fe}(- y) - \beta) - (A'_{fe}(- x) - \beta) A'_{ee}(- y).
\end{equation}
If $a_{ee} = 0$, then it holds trivially.
Otherwise, we rewrite \eqref{hee} by~\eqref{ee=hecoeff} as follows,
\begin{multline*}
0 = A'_{ee}(- x)(A'_{fe}(- y) - 2\zeta + 2\alpha a_{he}/a_{ee}) - (A'_{fe}(- x) - 2\zeta + 2\alpha a_{ahe}/a_{ee}) A'_{ee}(- y) \\
 = A'_{ee}(- x)(A'_{fe}(- y) - \beta) - (A'_{fe}(- x) - \beta) A'_{ee}(- y).
\end{multline*}

If $A'_{ee}\equiv0$ or $A'_{fe}-\beta\equiv0$, then the first equality of~\eqref{ee=fe} follows.
Otherwise, we get by~\eqref{ee-feKey} that
$$
\frac{A'_{ee}(-x)}{A'_{fe}(-x)-\beta} 
 = \frac{A'_{ee}(-y)}{A'_{fe}(-y)-\beta}
 = \mathrm{const},
$$
so, we also prove the first equality of~\eqref{ee=fe}.

Applying the same approach to the terms of~\eqref{hff}, with the help of Lemma~\ref{Mee} we get
$$
0 = (A'_{ef}(- x) - 4\zeta + \beta)A'_{ff}(-y) - A'_{ff}(-x)(A'_{fe}(- y) - 4\zeta + \beta),
$$
which implies the second equality of~\eqref{ee=fe}.

Evaluations of~\eqref{fee} and~\eqref{eff}
at $x + y = 0$ give us the second and the third equalities of~\eqref{ee=fecoeff}.
Whereas \eqref{fee} and~\eqref{eff} at $x = 0$ 
after the application of both equalities of~\eqref{ee=fe}
and the second and the third equalities of~\eqref{ee=fecoeff} give us 
\begin{equation}\label{TwoMinorsZero}
a_{ee}a_{hf} = a_{fe}a_{he}, \quad 
a_{ff}a_{he} = a_{fe}a_{hf}.
\end{equation}
Multiplying both relations, we get
$a_{he}a_{hf}(a_{ee}a_{ff} - a_{fe}^2) = 0$. 
If $a_{he},a_{hf}\neq0$, then we prove the first equality of~\eqref{ee=fecoeff}.
Otherwise, we compute \eqref{efh-shift} at $y = 0$ and obtain up to constant by~\eqref{cond-full}
\begin{equation} \label{lem8:help}
0 = A'_{ee}(-x)A'_{ff}(x) - (A'_{ef}(-x) - (4\zeta-\beta))(A'_{fe}(x)-\beta)
 -2(2\zeta-\beta)A'_{fe}(x).
\end{equation}
If $a_{ef} = 0$, then $a_{ee}a_{ff} = 0$, as required.
Otherwise, introduce $k = \deg(f_{ef})$. 
Then naturally $\deg(f_{ee}) + \deg(f_{ff}) = 2k$
and the coefficient at $x^{2k+2}$ in~\eqref{lem8:help} equals 
$0 = a_{ef}^2 - a_{ee}a_{ff}$.
Hence, we have proved the first equality of~\eqref{ee=fecoeff}.
\end{proof}

\begin{corollary} \label{coro:Main}
Let $L = \Cur(\sl_2(\mathbb{C}))$ and
$r = \sum A_{ql}(\partial_{\otimes 1}, 
\partial_{\otimes 2}) q\otimes l\in L\otimes L$, 
$q, l \in \{e, f, h\}$, be an $L$-invariant solution to the weak CCYBE. 
Then $A'_{ql}(x) - A'_{ql}(0) = a_{ql} x f(x^2)$ for some unitary polynomial $f(x)$
and $a_{ql}\in\mathbb{C}$ such that the matrix $M = (a_{ql})_{q,l=1}^3$ is a~symmetric matrix of rank at most 1.
\end{corollary}

\begin{proof}
Let us prove that there exists a unitary polynomial $f$ such that $f_{ql} = f$ for all $q,l=1,\ldots,3$. 
If all $a_{ql}$ are nonzero, it follows by~\eqref{ee=he},~\eqref{ff=hf},~\eqref{ee=fe}.

{\sc Case 1}: $a_{ee} = 0$.
Then $a_{he} = a_{ef} = 0$ by~\eqref{ee=hecoeff},~\eqref{ee=fecoeff}.
If at least one of the numbers $a_{ff},a_{fh},a_{hh}$ is zero, then it is easy to show that at most one of $a_{ql}$ is nonzero.
If all of $a_{ff},a_{fh},a_{hh}$ are nonzero, then $f_{ff} = f_{fh} = f_{hh}$ by~\eqref{ff=hf}.

{\sc Case 2}: $a_{hh} = 0$.
Then $a_{he} = a_{hf} = 0$ by~\eqref{ee=hecoeff},~\eqref{ee=fecoeff}.
We deal with this case analogously as with Case~1.

{\sc Case 3}: $a_{ef} = 0$.
By~\eqref{ee=fecoeff}, $a_{ee}a_{ff} = 0$.
Up to the action of~$\Psi$, we may assume that $a_{ee} = 0$, it is Case~1.

{\sc Case 4}: $a_{eh} = 0$. By~\eqref{ee=hecoeff}, $a_{ee}a_{hh} = 0$,
so we go to already considered Case~1 or~2.

Now, we want to prove that all minors of order two of~$M$ equal zero.  
By the first equalities of~\eqref{ee=hecoeff},~\eqref{ff=hfcoeff},~\eqref{ee=fecoeff} and by~\eqref{TwoMinorsZero} it remains to prove that $a_{ef}a_{hh} = a_{fh}a_{eh}$.
Consider~\eqref{efh-shift} at $x = 0$ and get up to constant
\begin{multline*}
0 = 2(A'_{ef}(-y)-(4\zeta-\beta))(A'_{hh}(-y)-\zeta)
 -2(A'_{eh}(-y)+\alpha)(A'_{fh}(-y)+\gamma) \\
 + 4\gamma A'_{eh}(-y) + (2\zeta-\beta)A'_{ef}(-y) + 2\alpha A'_{fh}(-y).
\end{multline*}
Analyzing the coefficient at $y^{2k+2}$, where $k = \deg f$, we get the required equality.
\end{proof}

\subsection{Description of solutions to the (weak) CCYBE}

Now we prove the main result of the work. 

\begin{theorem} \label{thm}
Let $L = \Cur(\sl_2(\mathbb{C}))$,
$r = \sum A_{ql}(\partial_{\otimes 1}, 
\partial_{\otimes 2}) q\otimes l\in L\otimes L$, 
$q, l \in \{e, f, h\}$, be an $L$-invariant solution to the weak CCYBE.
Then $A_{ql}(0,0)$ are defined by~\eqref{eq:A'ql(0)} for some $\alpha,\beta,\gamma,\zeta\in\mathbb{C}$. 
Moreover, $A_{ql}(x,-x) - A_{ql}(0,0) = a_{ql} x f(x^2)$ for a unitary polynomial 
$f(x)$ and $a_{q,l} \in \mathbb{C}$, and 
up to action of the automorphism group of~$L$, we have three cases:

(i) $a_{ee} = 1$, $a_{ql} = 0$ for $(q,l)\neq(e,e)$, and $\gamma = 2\zeta - \beta = 0$;

(ii) $a_{hh} = \lambda\in\mathbb{C}\setminus\{0\}$, $a_{ql} = 0$ for $(q,l)\neq(h,h)$, and $\alpha = \gamma = 0$; 

(iii) $a_{ql} = 0$ for all $q,l$.

\noindent
Conversely, every $r$ satisfying~\eqref{eq:A'ql(0)} and one of the conditions (i)--(iii) is a solution to the weak CCYBE.
\end{theorem}

\begin{proof}
Suppose that 
$r = \sum A_{ql}(\partial_{\otimes 1},\partial_{\otimes 2}) q\otimes l\in L\otimes L$, 
$q, l \in \{e, f, h\}$, is an $L$-invariant solution to the weak CCYBE on $L$.
For brevity, let us refer to $e, f, h$ as $1, 2, 3$, respectively.

By Corollary~\ref{coro:Main},
$A_{ql}(x,-x) - A_{ql}(0,0) = a_{ql} x f(x^2)$, for some unitary polynomial $f(x)$,
and $a_{ql}\in\mathbb{C}$ form a~symmetric matrix of rank at most 1.
Denote it $M = (a_{ql})$, $q,l\in \{1,2,3\}$.
Next, we simplify the matrix~$M$ with help of~$\Aut(L)$.  

Let $\Phi = (\Phi_{ij})$ be a matrix of an automorphism~$\varphi$ of $\sl_2(\mathbb{C})$
in the basis $e, f, h$. 
It is easy to compute that 
$$
\Phi = \begin{pmatrix} 
a^2 & -b^2 & -2ab \\ 
-c^2 & d^2 & 2cd  \\ 
-ac & bd & ad+bc\end{pmatrix},
$$
where $a, b, c, d \in \mathbb{C}$ satisfy the condition $ad - bc = 1$.
Then for any solution of the (weak) CCYBE we have 
\begin{multline*}
\varphi(r) 
 = \varphi\left(\sum_{q, l = 1}^3 A_{ql}(\partial_{\otimes 1},\partial_{\otimes 2}) q\otimes l\right) 
 = \sum_{q, l = 1}^3 A_{ql}(\partial_{\otimes 1},\partial_{\otimes 2}) \varphi(q)\otimes \varphi(l) 
 \allowdisplaybreaks \\
 = \sum_{q, l = 1}^3 \sum_{i, j = 1}^3
 A_{ql}(\partial_{\otimes 1}, \partial_{\otimes 2}) \Phi_{iq} i \otimes \Phi_{j l} j 
 = \sum_{i, j = 1}^3 \left(\sum_{q, l = 1}^3 \Phi_{iq} \Phi_{j l}
 A_{ql}(\partial_{\otimes 1}, \partial_{\otimes 2})\right)  (i \otimes  j).
\end{multline*}
Denote 
$\hat{A}_{ij}(x, y) = \sum_{q, l = 1}^3 \Phi_{iq} \Phi_{j l} A_{ql}(x, y)$. 
Hence, 
$$
\hat{A}_{ij}(x, y) - \hat{A}_{ij}(0, 0)
 = \sum_{q, l = 1}^3 \Phi_{iq} \Phi_{j l} (A_{ql}(x, y) - A_{ql}(0,0)),
$$
and we may assume that all constants are subtracted.
Due to the already proved conditions, let us consider
$$
\hat{a}_{ij} x \hat{f}(x^2) 
 = \hat{A}'_{ij}(x) 
 = \sum_{q, l = 1}^3 \Phi_{iq} \Phi_{j l} A'_{ql}(x) 
 = \sum_{q, l = 1}^3 \Phi_{iq} \Phi_{j l} a_{ql} x f(x^2),
$$
so, we conclude that $\hat{a}_{ij} = \sum_{q, l = 1}^3 \Phi_{iq} \Phi_{j l} a_{ql}$. 
In other words, $\hat{A} = \Phi A \Phi^T$ and the group $\Aut{(\sl_2(\mathbb{C}))}$ acts on the set of symmetric matrices of rank~1 by a congruence. 

Suppose that $a_{ii}\neq0$ for $i = 1,2,3$.
Take $\Phi$ with $b = 0$, then 
$\hat{a}_{11} = a^4 a_{11}$ and $\hat{a}_{22} = c^4 a_{11} - 4c^3d a_{13} +2c^2d^2(2a_{33}-a_{12}) + 4cd^3 a_{23} + d^4a_{22}$. Since $a_{11}\neq0$, we may choose $a,c$ such that $\hat{a}_{11} = \hat{a}_{22} = 1$.
Hence, $\hat{A}$ has one the following two forms:
\[
\begin{pmatrix}
1 & 1 & \pm \sqrt{a_{33}} \\
1 & 1 & \pm \sqrt{a_{33}} \\
\pm \sqrt{a_{33}} & \pm \sqrt{a_{33}} & a_{33}
\end{pmatrix}, \quad
\begin{pmatrix}
1 & -1 & \pm \sqrt{a_{33}} \\
-1 & 1 & \mp \sqrt{a_{33}} \\
\pm \sqrt{a_{33}} & \mp \sqrt{a_{33}} & a_{33}
\end{pmatrix}.
\]

For both variants, take $a = 0$, $b = 1$; 
$d = \pm \sqrt{a_{33}}c$ for the first one and 
$d = \mp \sqrt{a_{33}}c$ for the second one.
Then we get the matrix 
$Q = \begin{pmatrix}
1 & x & 0 \\
x & y & 0 \\
0 & 0 & 0
\end{pmatrix}$, 
where $y = x^2$.
If $x = 0$, then it is (i). 
Otherwise, we take $\Phi$ such that $a^2 = b^2 x$, $c^2 = d^2 x$ and $ad - bc = 1$, then 
$\hat{A} = \lambda e_{33}$, where $\lambda = 2bdx(bdx - ac)\in\mathbb{C}$.
It is the case (ii).

Now, consider the case when not all $a_{ii}$ are nonzero.
If only one of them is nonzero, then either $A = a_{33}e_{33}$ or $A = a_{ii}e_{ii}$ for $i\in\{1,2\}$. In the first case, we have (ii). In the second one, up to action of~$\Psi$ and $\Phi$ with $a^4 a_{11} = 1$, we get (i).
Again, up to action of~$\Psi$, it remains to study the case 
$A = \begin{pmatrix}
a_{11} & 0 & a_{13} \\
0 & 0 & 0 \\
a_{13} & 0 & a_{33}
\end{pmatrix}$,
where $a_{11}a_{33} = a_{13}^2$.
Then we apply action of $\Phi$ with $aca_{11} = (ad+bc)a_{13}$ and $c = aa_{13}$ to get the matrix~$Q$.

If $a_{ee}\neq0$, then $2\zeta-\beta = 0$ by~\eqref{ee=hecoeff} and $\gamma = 0$ by~\eqref{ee=fecoeff}.
If $a_{hh}\neq0$, then $\alpha = 0$ by~\eqref{ee=hecoeff} and $\gamma = 0$ by~\eqref{ff=hfcoeff}.

Conversely, let $r$ satisfy~\eqref{eq:A'ql(0)} and one of the conditions (i)--(iii).
In \S\ref{sec:Skew-sym}, we have found an equivalent conditions for $L$-invariance of~$r$. They, i.\,e. \eqref{cond-full}, are fulfilled for indicated~$r$.
Fixing $a\in\sl_2(\mathbb{C})$,
we compute $a_{\mu}\llbracket r,r\rrbracket$
by~\eqref{rr} modulo $\mu = - \partial^{\otimes 3}$:
\begin{multline*}
\sum_{q,q',l,l' = 1}^3 (A'_{ql}(-\partial_{\otimes 2})A'_{q'l'}(-\partial_{\otimes 3}) 
[a, [q, q']] \otimes l \otimes l'
+ A'_{ql}(\partial_{\otimes 1}+\partial_{\otimes 3})
A'_{q'l'}(-\partial_{\otimes 3})[q, q'] \otimes [a, l] \otimes l' \\
+ A'_{ql}(-\partial_{\otimes 2})A'_{q'l'}(\partial_{\otimes 1}+\partial_{\otimes 2})[q, q'] \otimes l \otimes [a, l']
- A'_{ql}(-\partial_{\otimes 2}-\partial_{\otimes 3})A'_{q'l'}(-\partial_{\otimes 3})[a, q] \otimes [q', l] \otimes l' \\
- A'_{ql}(\partial_{\otimes 1})A'_{q'l'}(-\partial_{\otimes 3})
q \otimes [a, [q', l]] \otimes l' 
- A'_{ql}(\partial_{\otimes 1})A'_{q'l'}(\partial_{\otimes 1}+\partial_{\otimes 2})
q \otimes [q', l] \otimes [a, l'] \\
- A'_{ql}(-\partial_{\otimes 2}-\partial_{\otimes 3})A'_{q'l'}(\partial_{\otimes 2})[a, q] \otimes q' \otimes [l', l] 
- A'_{ql}(\partial_{\otimes 1})A'_{q'l'}(-\partial_{\otimes 1}-\partial_{\otimes 3})q \otimes [a, q'] \otimes [l', l] \\ 
- A'_{ql}(\partial_{\otimes 1})A'_{q'l'}(\partial_{\otimes 2}) q \otimes q' \otimes [a, [l', l]]).
\end{multline*}
By Lemma~\ref{lem:ZeroLevel},
the constant coefficient (not depending on $\partial$) at $q\otimes l\otimes m$ for all $q,l,m\in \{1,2,3\}$ equals zero.
Thus, we are done with the case (iii).
Note that when $q = l = q' = l'$, we get zero due to the multiplication table in~$\sl_2(\mathbb{C})$. 

Сonsider the case (i). 
Recall that $a_{ee} = 1$, $A'_{he}(0) = - A'_{eh}(0) = \alpha$, $A'_{fe}(0) = A'_{fe}(0) = 2A'_{hh} = \beta$, and $A'_{ql} \equiv 0$ for all other $q,l$.
Let us gather terms according to a common factor $A'_{11}(t)$, for 
$t = \partial_{\otimes i}$ or $t = \partial_{\otimes i}+\partial_{\otimes j}$, $i\neq j$.
Say, at $A'_{11}(\partial_{\otimes 1})$ we have
\begin{multline*}
- \sum_{q',l' = 1}^3 A'_{q'l'}(0)( 1 \otimes [a, [q', 1]] \otimes l' 
+ 1 \otimes [q', 1] \otimes [a, l'] 
+ 1 \otimes [a, q'] \otimes [l', 1] 
+ 1 \otimes q' \otimes [a, [l', 1]]) \\
 = - 1\otimes a_{\mu}\left(\sum_{q',l' = 1}^3 A'_{q'l'}(0)(
 [q',1]\otimes l' + q'\otimes [l',1])
\right) \\
 = 1\otimes a_{\mu}\left(1_{\mu}\left(\sum_{q',l' = 1}^3 A'_{q'l'}(0)q'\otimes l'\right)\right) 
 = 1\otimes a_{\mu} 0 = 0
\end{multline*}
Analogously, we deal with all other sums with a common factor $A'_{11}(t)$. Hence, $r$ is a~solution to the weak CCYBE. 

In the case (ii), when $a_{hh} = \lambda\neq0$, $A'_{fe}(0) = \beta$, $A'_{ef}(0) = 2\zeta-\beta$, $A'_{hh}(0) = \zeta$, and $A'_{ql} \equiv 0$ for all other $q,l$, we apply the same approach to show that $r$ is a solution to the weak CCYBE.
\end{proof}

\begin{corollary} \label{coro:main}
Let $L = \Cur(\sl_2(\mathbb{C}))$,
$r = \sum A_{ql}(\partial_{\otimes 1}, 
\partial_{\otimes 2}) q\otimes l\in L\otimes L$, 
$q, l \in \{e, f, h\}$, be a skew-symmetric solution to the CCYBE. 
Then $A_{ql}(0,0)$ are defined by~\eqref{eq:A'ql(0)} with $\zeta = 0$. 
Moreover, $A_{ql}(x,-x) - A_{ql}(0,0) = a_{ql} x f(x^2)$ for some unitary polynomial $f(x)$ and $a_{q,l} \in \mathbb{C}$, and 
up to action of the automorphism group of~$L$, we have three cases:

(i) $a_{ee} = 1$, $a_{ql} = 0$ for $(q,l)\neq(e,e)$, and $\beta = \gamma = 0$;

(ii) $a_{hh} = \lambda\in\mathbb{C}\setminus\{0\}$, $a_{ql} = 0$ for $(q,l)\neq(h,h)$, and $\alpha =  \beta = \gamma = 0$; 

(iii) $a_{ql} = 0$ for all $q,l$.
\end{corollary}

\begin{proof}
The condition $\zeta = 0$ implies that $r$ is a~skew-symmetric solution to CCYBE.
It remains to take a solution to the weak CCYBE from Theorem~\ref{thm} and substitute it into~\eqref{efh}. 
Thus, we get the additional condition $\beta = 0$ in (ii).
\end{proof}

\section{Solutions to the (weak) CCYBE on $\mathrm{Vir}$}

Now, let us shed some light on the situation with $\mathrm{Vir}$. 
Recall that $\mathrm{Vir} = \mathbb{C}[\partial]v$ with the multiplication rule $[v_{\lambda} v] = 
(\partial + 2 \lambda)v$. 

Consider $r = A(\partial_{\otimes 1}, \partial_{\otimes 2}) v \otimes v$ for some
$A(x, y)\in \mathbb{C}[x,y]$.
Analogously to~\S3.2, $r$~is $L$-invariant if and only if
the projection of $A(x, y)$ on  $x + y = 0$ is  skew-symmetric.
Indeed, consider $v_\mu (r + \tau(r)) = 0$ modulo $\mu = - \partial^{\otimes 2}$,
we follow the notation introduced in~\S3:
\begin{multline*}
0 = v_\mu (A(\partial_{\otimes 1}, \partial_{\otimes 2}) + A(\partial_{\otimes 2}, \partial_{\otimes 1})) v \otimes v \\
=  ((A(\partial_{\otimes 1} + \mu, \partial_{\otimes 2}) + A(\partial_{\otimes 2}, \partial_{\otimes 1} + \mu))(\partial_{\otimes 1} + 2\mu) \\
+ (A(\partial_{\otimes 1}, \partial_{\otimes 2} + \mu) + A(\partial_{\otimes 2} + \mu, \partial_{\otimes 1}))(\partial_{\otimes 2} + 2\mu))
v \otimes v  \\
=  ((A'(- \partial_{\otimes 2}) + A'(\partial_{\otimes 2}))(-\partial_{\otimes 1} - 2\partial_{\otimes 2})
+ (A'(\partial_{\otimes 1}) + A'(- \partial_{\otimes 1}))(-\partial_{\otimes 2} - 2\partial_{\otimes 1})) v \otimes v.
\end{multline*}
Evaluating at $\partial_{\otimes 2} = 0$, we get that $A(x,-x) = -A(-x,x)$.

Next, suppose that $r$ is an $L$-invariant solution to the weak CCYBE. Consider 
$v_{\mu}\llbracket r,r\rrbracket$ modulo $\mu = - \partial^{\otimes 3}$:
\begin{multline*}
0 = -[A'(\partial_{\otimes 2})A'(\partial_{\otimes 3})(\partial_{\otimes 2} - \partial_{\otimes 3}) - A'(\partial_{\otimes 2} 
+ \partial_{\otimes 3})A'(\partial_{\otimes 3})(\partial_{\otimes 2} + 2\partial_{\otimes 3}) \\
+ A'(\partial_{\otimes 2} + \partial_{\otimes 3})A'(\partial_{\otimes 2})(\partial_{\otimes 3} + 2 \partial_{\otimes 2})]
(\partial_{\otimes 1} + 2\partial_{\otimes 2} + 2\partial_{\otimes 3}) \\
- [A'(\partial_{\otimes 1} + \partial_{\otimes 3})A'(\partial_{\otimes 3})(\partial_{\otimes 1} + 2\partial_{\otimes 3}) 
+ A'(\partial_{\otimes 1})A'(\partial_{\otimes 3})(\partial_{\otimes 3} - \partial_{\otimes 1}) \\
- A'(\partial_{\otimes 1})
A'(\partial_{\otimes 2})(\partial_{\otimes 3} + 2\partial_{\otimes 1})](2\partial_{\otimes 1} + \partial_{\otimes 2} +
2\partial_{\otimes 3}) \\
- [-A'(\partial_{\otimes 2})A'(\partial_{\otimes 1} + \partial_{\otimes 2})(\partial_{\otimes 1} + 2\partial_{\otimes 2}) + 
A'(\partial_{\otimes 1})A'(\partial_{\otimes 1} + \partial_{\otimes 2})(\partial_{\otimes 2} + 2\partial_{\otimes 1}) \\
- A'(\partial_{\otimes 1})A'(\partial_{\otimes 2})(\partial_{\otimes 2} - \partial_{\otimes 1})](2\partial_{\otimes 1} + 2\partial_{\otimes 2} + 
\partial_{\otimes 3}).
\end{multline*}
Take $\partial_{\otimes 3} = 0$ and $\partial_{\otimes 1} + 2 \partial_{\otimes 2} = 0$.
Then the first sum is zero,
the remaining two sums give the equality,
$$
- 12 A'(2\partial_{\otimes 2})A'(\partial_{\otimes 2})\partial_{\otimes 2}^2 = 0.
$$
The only polynomial which fulfills this equality is zero one.
Hence, the necessary condition is that $A(x,-x) = 0$.
Sufficiency is obvious.

\begin{theorem} \label{Vir}
Let $L = \mathrm{Vir}$,
$r = A(\partial_{\otimes 1}, \partial_{\otimes 2}) v\otimes v\in L\otimes L$.
Then $r$ is an $L$-invariant solution to the (weak) CCYBE if and only if $A(x, -x) = 0$.
\end{theorem}

\section*{Acknowledgements}

The authors are supported by the grant of the President of the Russian
Federation for young scientists (MK-1241.2021.1.1).

\medskip
\noindent Vsevolod Gubarev \\
Roman Kozlov \\
Sobolev Institute of Mathematics \\
Acad. Koptyug ave. 4, 630090 Novosibirsk, Russia \\
Novosibirsk State University \\
Pirogova str. 2, 630090 Novosibirsk, Russia \\
e-mail: wsewolod89@gmail.com, dyadkaromka94@gmail.com

\end{document}